\documentclass{amsart}
\usepackage{amssymb}

\newtheorem{theorem}{Theorem}[section]
\newtheorem{lemma}[theorem]{Lemma}
\newtheorem{proposition}[theorem]{Proposition}
\theoremstyle{remark}
\newtheorem*{remarks}{Remarks}
\newtheorem{remark}[theorem]{Remark}

\numberwithin{equation}{section}
\begin{document}

\title[Combinatorial bases of basic modules for  $C_{n}\sp{(1)}$]
{Combinatorial bases of basic modules for affine Lie algebras $C_{n}\sp{(1)}$}
\author{Mirko Primc and Tomislav \v Siki\' c}
\address{Mirko Primc, University of Zagreb, Faculty of Science, Bijeni\v cka 30, 10000 Zagreb, \mbox{\hskip 7.6em} Croatia}
\email{primc@math.hr}
\address{Tomislav \v{S}iki\'{c}, University of Zagreb, Faculty of Electrical Engineering and Com- \mbox{\hskip 8.5em} puting, Unska 3, 10000 Zagreb,  Croatia}
\email{tomislav.sikic@fer.hr}

\subjclass[2000]{Primary 17B67; Secondary 17B69, 05A19.\\
Partially supported by the Croatian Science Foundation under the project 2634 and by the\\
Croatian Scientific Centre of Excellence QuantixLie.}

\begin{abstract}

J.~Lepowsky and R.~L.~Wilson initiated the approach to combinatorial
Rogers-Ramanujan type identities via vertex operator constructions
of standard (i.e. integrable highest weight) representations of
affine Kac-Moody Lie algebras. A.~Meurman and M.~Primc developed
further this approach for $\mathfrak{sl}(2,\mathbb C)\widetilde{}\ $
by using vertex operator algebras and Verma modules. In this paper
we use the same method to construct combinatorial bases of basic
modules for affine Lie algebras of type $C_{n}\sp{(1)}$ and, as a
consequence, we obtain a series of Rogers-Ramanujan type identities.
A major new insight is a combinatorial parametrization of leading
terms of defining relations for level one standard modules for affine Lie
algebra of type $C_{n}\sp{(1)}$.
\end{abstract}
\maketitle

\section{Introduction}

J.~Lepowsky and R.~L.~Wilson \cite{LW} initiated the approach to
combinatorial Rogers-Ramanujan type identities via vertex operator
constructions of representations of affine Kac-Moody Lie algebras.
In \cite{MP1} this approach is  developed further for
$\mathfrak{sl}(2,\mathbb C)\widetilde{}\ $ by using vertex operator
algebras and Verma modules. In this paper we use the same method to
construct combinatorial bases for basic modules of affine Lie
algebra of type $C_{n}\sp{(1)}$.

The starting point in \cite{MP1} is a PBW spanning set of a standard
(i.e., integrable highest weight) module $L(\Lambda)$ of level $k$,
which is then reduced to a basis by using the relation
$$
x_\theta(z)^{k+1}=0\quad\text{on}\quad L(\Lambda).
$$
In \cite{MP1} this relation was interpreted in terms of vertex
operator algebras and it was proved for any level $k$ standard
module of any untwisted affine Kac-Moody Lie algebra.

After a PBW spanning set is reduced to a basis, it remains to prove
its linear independence. The main ingredient of the proof is a
combinatorial use of relation
$$
x_\theta(z)\tfrac{d}{dz}(x_\theta(z)^{k+1})=
(k+1)x_\theta(z)^{k+1}\tfrac{d}{dz}x_\theta(z)
$$
for the annihilating field $x_\theta(z)^{k+1}$. This relation was
also interpreted in terms of vertex operator algebras.

By following ideas developed in \cite{MP1} and \cite{MP2}, in
\cite{P1} and \cite{P2} a general construction of relations for
annihilating fields is given by using vertex operator algebras, and
by using these relations the problem of constructing combinatorial
bases of standard modules is split into a ``combinatorial part of
the problem'' and a ``representation theory part of the problem ''.

In this paper we use these results to construct combinatorial bases
of basic modules for affine Lie algebras of type $C_{n}\sp{(1)}$.  A
major new insight is a combinatorial parametrization in  \cite{PS} of leading
terms of defining relations for all standard modules for affine Lie
algebra of type $C_{n}\sp{(1)}$. This is, hopefully, an important
step towards a solution of ``combinatorial part of the problem'' of
constructing combinatorial bases of standard modules for affine Lie
algebras.

In first nine sections we give a detailed outline of ideas and
results involved in this approach, we introduce notation and recall
necessary general results from \cite{P1} and \cite{P2}. The results
from \cite{P1} on relations among relations are formulated in
``untwisted setting''---this may alleviate using the results which are quite
technical in ``twisted setting''. In Section~10 we prove Proposition~\ref{P:10.1}
which is the starting point of our construction of combinatorial
basis of the basic module $L(\Lambda_0)$ for affine Lie algebra of
type $C_{n}\sp{(1)}$. In Section~11 we prove linear independence of
combinatorial bases by using the combinatorial result from \cite{PS}
for counting the number of two-embeddings. As a consequence, in Section 12
we obtain a series of combinatorial Rogers-Ramanujan type identities.

We thank Arne Meurman for many stimulating discussions and help in understanding the combinatorics of leading terms.

\section{Vertex algebras and generating fields}

Two formal Laurent series $a(z)=\sum a_n z^{-n-1}$ and $b(z)=\sum
b_n z^{-n-1}$, with coefficients in some associative algebra, are
said to be mutually local if for some non-negative integer $N$
$$
(z_1-z_2)^Na(z_1)b(z_2)=(z_1-z_2)^Nb(z_2)a(z_1).
$$
A vertex algebra $V$ is a vector space equipped with a specified
vector $\mathbf 1 $ called the vacuum vector, a linear operator
$D$ on $V$ called the derivation and a linear map
$$
V\rightarrow (\text{End}\,V)[[z^{-1},z]],\ \ \ \ v\mapsto
Y(v,z)=\sum_{n\in\mathbb Z }v_nz^{-n-1}
$$
satisfying the following conditions for $u,v\in V$:
\begin{eqnarray}
& &u_nv=0\qquad
\text{for}\ n \ \text{sufficiently\ large},\label{E:2.1}\\
& &[D,Y(u,z)]=Y(Du,z)=\frac{d}{dz}Y(u,z),\label{E:2.2}\\
& &Y({\bf 1},z)=\text{id}_V\qquad (\text{the identity operator on}\ V),\label{E:2.3}\\
& &Y(u,z)\mathbf 1  \in (\text{End}\,V)[[z]]\quad \text{and}\quad
\lim_{z\to 0}Y(u,z)\mathbf 1  =u, \label{E:2.4}\\
& &Y(u,z)\text{\ and\ } Y(v,z) \text{\ are mutually
local.}\label{E:2.5}
\end{eqnarray}

Haisheng Li showed \cite{L} that this definition of vertex algebra
is equivalent to the original one given by R.~E.~Borcherds
\cite{B}. A formal Laurent series $Y(u,z)$ is called the vertex
operator (field) associated with the vector (state) $u$, and
(\ref{E:2.4}) gives a state-field correspondence. For coefficients
of vertex operators $Y(u,z)$ and $Y(v,z)$ we have the commutator
formula
\begin{equation}\label{E:2.6}
[u_m,v_n]=\sum_{i\geq 0}\binom{m}{i}(u_i v)_{m+n-i}.
\end{equation}

Let $M$ be a vector space and  $a(z)$ and $b(z)$ two formal
Laurent series with coefficients in $\text{End}\,M$ such that for
each $w\in M$
\begin{equation}\label{E:2.7}
a_m w=0\quad  \text{and} \quad b_m w=0 \qquad \text{for\ }  m
\text{\ sufficiently\ large.}
\end{equation}
Then for each integer $n$ we have a well defined product
\begin{equation}\label{E:2.8}
   a(z)_nb(z)=\text{Res}_{z_1}\left((z_1-z)^n a(z_1)b(z)-
(-z+z_1)^n b(z)a(z_1)\right),
\end{equation}
with the convention that $(z_1-z)^n=z_1^n(1-z/z_1)^n$ denotes a
series obtained by the binomial formula for $(1-\zeta)^n$. If we
think of vertex algebra as a vector space given $\mathbf 1 $, $D$
and multiplications $u_nv$, satisfying
(\ref{E:2.1})--(\ref{E:2.5}), then we can state the theorem on
generating fields due to Haisheng Li \cite{L}:
\begin{theorem}\label{T:2.1}
A family of mutually local formal Laurent series with coefficients
in $\text{End}\,M$, satisfying (\ref{E:2.7}), generates a vertex
algebra with the vacuum $\mathbf 1 =\text{id}_M$, the derivation
$D=\frac{d}{dz}$ and the multiplications $a(z)_nb(z)$.
\end{theorem}

Vertex operator algebra (see \cite{FLM}) is a vertex algebra $V$
with a conformal vector $\omega$ such that $Y(\omega,z)=\sum
L_nz^{-n-2}$ gives the Virasoro algebra operators $L_n\,$, with
$L_{-1}=D$. It is also required that $L_0$ defines a $\mathbb
Z$-grading $V=\coprod V_n$ truncated from below with
finite-dimensional eigenspaces $V_n$.

For $u\in V_n$ we write $\text{wt\,}u=n$. We shall sometimes use
another convention for writing coefficients of vertex operators,
$$
Y(u,z)=\sum_{n\in\mathbb Z }u(n)z^{-n-\text{wt\,}u},
$$
so that $u(n)$ is a homogeneous operator on the graded space $V$
of degree $n$.

For a vertex operator algebra $V$ we have a vertex operator
algebra structure on $V\otimes V$ with fields $Y(u\otimes
v,z)=Y(u,z)\otimes Y(v,z)$ and the conformal vector $\omega\otimes
\mathbf 1 +\mathbf 1 \otimes\omega$ (see \cite{FHL}).

\section{Vertex algebras for affine Lie algebras}

Let ${\mathfrak g}$ be a simple complex Lie algebra, $\mathfrak h$
a Cartan subalgebra of ${\mathfrak g}$ and $\langle \ , \ \rangle$
a symmetric invariant bilinear form on ${\mathfrak g}$. Via this
form we identify $\mathfrak h$ and $\mathfrak h^*$ and we assume
that $\langle \theta , \theta \rangle=2$ for the maximal root
$\theta$ (with respect to some fixed basis of the root system).
Set
$$
\hat{\mathfrak g} =\coprod_{j\in\mathbb Z}{\mathfrak g}\otimes
t^{j}+\mathbb C c, \qquad \tilde{\mathfrak g}=\hat{\mathfrak
g}+\mathbb C d.
$$
Then $\tilde{\mathfrak g}$ is the associated untwisted affine
Kac-Moody Lie algebra (cf. \cite{K}) with the commutator
$$
[x(i),y(j)]=[x,y](i+j)+i\delta_{i+j,0}\langle x,y\rangle c.
$$
Here, as usual, $x(i)=x\otimes t^{i}$ for $x\in{\mathfrak g}$ and
$i\in\mathbb Z$, $c$ is the canonical central element, and
$[d,x(i)]=ix(i)$. Sometimes we shall denote ${\mathfrak g}\otimes
t^{j}$ by ${\mathfrak g}( {j})$. We identify ${\mathfrak g}$ and
${\mathfrak g}(0)$. Set
$$
\tilde{\mathfrak g}_{<0} =\coprod_{j<0}{\mathfrak g}\otimes
t^{j},\qquad \tilde{\mathfrak g}_{\leq 0} =\coprod_{j\leq
0}{\mathfrak g}\otimes t^{j}+\mathbb C d,\qquad\tilde{\mathfrak
g}_{\geq 0} =\coprod_{j\geq 0}{\mathfrak g}\otimes t^{j}+\mathbb C
d.
$$

For $k\in\mathbb C$ denote by $\mathbb C v_k$ the one-dimensional
$(\tilde{\mathfrak g}_{\geq 0}+\mathbb C c)$-module on which
$\tilde{\mathfrak g}_{\geq 0}$ acts trivially and $c$ as the
multiplication by $k$. The affine Lie algebra $\tilde{\mathfrak
g}$ gives rise to the vertex operator algebra (see \cite{FZ} and
\cite{L}, here we use the notation from \cite{MP1})
$$
N(k\Lambda_0)=U(\tilde{\mathfrak
g})\otimes_{U(\tilde{\mathfrak g}_{\geq 0}+\mathbb C c)}\mathbb C
v_k
$$
for level $k\neq -g^\vee$, where $g^\vee$ is the dual Coxeter
number of ${\mathfrak g}$; it is generated by the fields
\begin{equation}\label{E:3.1}
x(z)=\sum_{n\in\mathbb Z}x_nz^{-n-1},\qquad x\in{\mathfrak g},
\end{equation}
where we set $x_n=x(n)$ for $x\in{\mathfrak g}$. By the
state-field correspondence we have
$$
x(z)=Y(x(-1)\mathbf 1 ,z)\quad\text{for}\ x\in{\mathfrak g}.
$$
The $\mathbb Z$-grading is given by $L_0=-d$.

From now on we shall fix the level $k\in\mathbb Z_{>0}$, and we
shall often denote by $V$ the vertex operator algebra structure on
the generalized Verma $\tilde{\mathfrak g}$-module
$N(k\Lambda_0)$.

\section{A completion of the enveloping algebra}

Let $\mathcal U=U(\hat{\mathfrak g})/(c-k)$, where
$U(\hat{\mathfrak g})$ is the universal enveloping algebra of
$\hat{\mathfrak g}$ and $(c-k)$ is the ideal generated by the
element $c-k$. Note that $\tilde{\mathfrak g}$-modules of level
$k$ are $\mathcal U$-modules. Note that $U(\hat{\mathfrak g})$ is
graded by the derivation $d$, and so is the quotient $\mathcal U$.
Let us denote the homogeneous components of the graded algebra
$\mathcal U$ by $\mathcal U(n)$, $n\in \mathbb Z$. We take
\begin{equation}\label{E:4.1}
W_p(n)=\sum_{i\geq p} \mathcal U(n-i)\mathcal U(i),\qquad
p\in\mathbb Z_{>0} \ ,
\end{equation}
to be a fundamental system of neighborhoods of $0\in \mathcal
U(n)$. It is easy to see that we have a Hausdorff topological
group $(\mathcal U(n),+)$, and we denote by $\overline{\mathcal
U(n)}$ the corresponding completion, introduced in \cite{FZ} (cf.
also \cite{H}, \cite{FF}, and \cite{KL}). Then
$$
\overline{\mathcal U}=\coprod_{n\in \mathbb Z}\overline{\mathcal
U(n)}
$$
is a topological ring.

The definition (\ref{E:4.1}) of a fundamental system of
neighborhoods is so designed that the product $a(z)_nb(z)$ of two
formal Laurent series with coefficients in $\overline{\mathcal U}$
is well defined by the formula (\ref{E:2.8}). Haisheng Li's
arguments in the proof of Theorem~\ref{T:2.1} apply literally and
we have:
\begin{proposition}\label{P:4.1}
A family of mutually local formal Laurent series (\ref{E:3.1})
with coefficients in $\overline{\mathcal U}$ generates a vertex
algebra $V'$ with the vacuum $1\in \overline{\mathcal U}$, the
derivation $D=\frac{d}{dz}$ and the multiplications $a(z)_nb(z)$.
Moreover, the linear map
$$
Y\colon x(-1)\mathbf 1 \mapsto x(z)\qquad\text{for}\
x\in{\mathfrak g}
$$
extends uniquely to an isomorphism \ $Y\colon V\to V'$  of vertex
operator algebras.
\end{proposition}
The map
\begin{equation*}
Y \colon V \rightarrow \overline{\mathcal U}\,[[z,z^{-1}]],\qquad
v\mapsto Y(v,z)=\sum_{n\in\mathbb Z}v_nz^{-n-1},
\end{equation*}
was first constructed by I.~B.~Frenkel and Y.~Zhu in
\cite[Definition 2.2.2]{FZ} by using another method. From now on
we shall consider the coefficients $v_n$ of $Y(v,z)$ for $v\in V$
as elements in the completion $\overline{\mathcal U}$. Then for
any highest weight $\tilde{\mathfrak g}$-module $M$ of level $k$
the elements $v_n\in\overline{\mathcal U}$ act on $M$, defining a
representation of the vertex operator algebra $V$ on $M$.

By following the notation in \cite{FF} we set
$$
U_{\text{loc}}=\mathbb C\text{-span}\{v_n \mid v \in V, n \in
\mathbb Z\}\subset\overline{\mathcal U},
$$
 {}From the
commutator formula (\ref{E:2.6}) we see that $U_{\text{loc}}$ is a
Lie algebra. Let us denote by $U$ the associative subalgebra of
$\overline{\mathcal U}$ generated by $U_{\text{loc}}$. By
construction we have $\mathcal U\,\subset U$. Clearly
$$
U=\coprod_{n\in\mathbb Z} U(n),
$$
where $U(n)\subset U$ is the homogeneous subspace of degree $n$.

\section{Annihilating fields of standard modules}

For the fixed positive integer level $k$ the generalized Verma
$\tilde{\mathfrak g}$-module $N(k\Lambda_0)$ is  reducible, and we
denote by $N^1(k\Lambda_0)$ its maximal $\tilde{\mathfrak
g}$-submodule. By \cite[Corollary 10.4]{K} the submodule
$N^1(k\Lambda_0)$ is generated by the singular vector
$x_\theta(-1)^{k+1}\mathbf 1$, where $x_\theta$ is a root vector
in $\mathfrak g$. Set
$$
R=U(\mathfrak g)x_\theta(-1)^{k+1}\mathbf 1,\qquad \bar R =
\mathbb C\text{-span}\{r_n \mid r \in R, n \in \mathbb Z\}.
$$
Then $R\subset N^1(k\Lambda_0)$ is an irreducible $\mathfrak
g$-module, and $\bar R\subset U$ is the corresponding loop
$\tilde{{\mathfrak g}}$-module for the adjoint action given by the
commutator formula (\ref{E:2.6}).

We have the following theorem (see \cite{DL}, \cite{FZ}, \cite{L},
\cite{MP1}):
\begin{theorem}\label{T:5.1}
 Let $M$ be a highest weight $\tilde{\mathfrak
g}$-module of level $k$. The following are equivalent:
\begin{enumerate}
\item $M$ is a standard module,
\item $\bar R$ annihilates $M$.
\end{enumerate}
\end{theorem}
This theorem implies that for a dominant integral weight
${\Lambda}$ of level ${\Lambda}(c) = k$ we have
$$
\bar RM({\Lambda}) = M^1({\Lambda}),
$$
where $M^1({\Lambda})$ denotes the maximal submodule of the Verma
$\tilde{\mathfrak g}$-module $M({\Lambda})$. Furthermore, since
$R$ generates the vertex algebra ideal $N^1(k\Lambda_0)\subset V$,
vertex operators $Y(v,z)$, $v\in N^1(k\Lambda_0)$, annihilate all
standard $\tilde{\mathfrak g}$-modules
$L({\Lambda})=M({\Lambda})/M^1({\Lambda})$ of level $k$.

We shall call the elements $r_n\in\bar R$ {\it relations} (for
standard modules), and $Y(v,z)$, $v\in N^1(k\Lambda_0)$, {\it
annihilating fields} (of standard modules). It is clear that the
field
$$
Y(x_\theta(-1)^{k+1}\mathbf 1,z)=x_\theta(z)^{k+1}
$$
generates all annihilating fields.

\section{Tensor products and induced representations}

The vertex operator algebra $V$ has a Lie algebra structure with
the commutator
\begin{equation}\label{E:6.1}
[u,v]=u_{-1}v-v_{-1}u=\sum_{n\geq 0}(-1)^nD^{(n+1)}(u_nv),
\end{equation}
and $\tilde{\mathfrak g}_{<0}\mathbf 1$ is a Lie subalgebra.
Moreover, the map
$$
\tilde{\mathfrak g}_{<0}\mathbf 1\to\tilde{\mathfrak
g}_{<0},\qquad u\mapsto u_{-1},
$$
is a Lie algebra isomorphism and we have the ``adjoint'' action
$$
u_{-1}\colon v\mapsto [u,v]
$$
of the Lie algebra $\tilde{\mathfrak g}_{<0}$ on $V$. Since
$L_{-1}$, $L_{0}$ and $y_0$, $y\in {\mathfrak g}$, are derivations
of the product $u_{-1}v$, they are also derivations of the bracket
$[u,v]$, and we can extend the ``adjoint'' action of the Lie
algebra $\tilde{\mathfrak g}_{<0}$ on $V$ to the ``adjoint''
action of the Lie algebra
$$
\mathbf C L_{-1}\ltimes\tilde{\mathfrak g}_{\leq 0}\cong
\big(\mathbf C L_{-1}+\mathbf C L_{0}+{\mathfrak g}(0)\big)\ltimes
\tilde{\mathfrak g}_{<0}\mathbf 1.
$$

The subspace
$$
\bar R\,\mathbf 1=\coprod_{i=0}^n D^iR \subset V
$$
is a $\tilde{\mathfrak g}_{\geq 0}$-submodule invariant for the
action of $D=L_{-1}$. Then the right hand side of (\ref{E:6.1})
implies that $\bar R\,\mathbf 1$ is invariant for the ``adjoint''
action of $\mathbf C L_{-1}\ltimes\tilde{\mathfrak g}_{\leq 0}$,
we shall denote it by $(\bar R\,\mathbf 1)_\text{ad}$.

Hence we have the induced $\tilde{\mathfrak g}$-module
$U(\tilde{\mathfrak g})\otimes_{U(\tilde{\mathfrak g}_{\geq
0}+\mathbb C c)}\bar R\,\mathbf 1$ and the tensor product $(\bar
R\,\mathbf 1)_\text{ad}\otimes V$ of $(\mathbf C
L_{-1}\ltimes\tilde{\mathfrak g}_{\leq 0})$-modules, and we have
two maps
\begin{equation*}
\begin{array}{ll}
\Psi \colon U(\tilde{\mathfrak g})\otimes_{U(\tilde{\mathfrak
g}_{\geq 0}+\mathbb C c)} \bar R\,\mathbf 1
\to N(k\Lambda_0),\qquad &u\otimes w \mapsto uw,\\
{}&{}\\
\Phi \colon (\bar R\,\mathbf 1)_\text{ad} \otimes V\to V,\qquad
&u\otimes w \mapsto u_{-1}w.
\end{array}
\end{equation*}
Note that the map $\Psi $ is a homomorphism of $\tilde{\mathfrak
g}$-modules, and that $\Psi$ intertwines the actions of $L_{-1}$
and $L_{0}$. Hence, by restriction, $\Psi$ is a $(\mathbf C
L_{-1}\ltimes\tilde{\mathfrak g}_{\leq 0})$-module map. The
following theorem relates $\ker \Phi$ with induced representations
of $\tilde{\mathfrak g}$:

\begin{theorem}\label{T:6.1} (i) There is a unique isomorphism of \ $(\mathbf C
L_{-1}\ltimes\tilde{\mathfrak g}_{\leq 0})$-modules
$$
\Xi \colon (\bar R\,\mathbf 1)_\text{ad}\otimes V\to
U(\tilde{\mathfrak g})\otimes_{U(\tilde{\mathfrak g}_{\geq
0}+\mathbb C c)} \bar R\,\mathbf 1
$$
such that \  $\Xi(w\otimes \mathbf 1)= 1\otimes w$ \ for all $w\in
\bar R\,\mathbf 1$.

(ii) The map $\Phi$ is is a homomorphism of $(\mathbf C
L_{-1}\ltimes\tilde{\mathfrak g}_{\leq 0})$-modules and
$\Phi=\Psi\circ \Xi$. In particular, $\ker \Phi$ is a $(\mathbf C
L_{-1}\ltimes\tilde{\mathfrak g}_{\leq 0})$-module and
$$
\Xi(\ker \Phi)=\ker \Psi.
$$
\end{theorem}

We call elements in $\ker \Phi$\, {\it relations for annihilating
fields} (cf. \cite{P1}, \cite{P2}) since
$$
\sum :Y(a,z)Y(b,z):\, =0\qquad\text{for}\qquad \sum a\otimes b\in
\ker \Phi.
$$
By Theorem~\ref{T:6.1} we may identify the relations for
annihilating fields with elements of $\ker \Psi$, which is easier
to study by using the representation theory of affine Lie
algebras.

\section{Generators of relations for annihilating fields}

Let $\{x^i\}_{i\in I}$ and $\{y^i\}_{i\in I}$ be dual bases in
$\mathfrak g$. For $r\in R$ we define Sugawara's relation
\begin{equation}\label{E:7.1}
q_r=\frac{1}{k+g^\vee}\sum_{i\in I}x^i(-1)\otimes y^i(0)r-1\otimes
Dr
\end{equation}
as an element of $U(\tilde{\mathfrak
g})\otimes_{U(\tilde{\mathfrak g}_{\geq 0}+\mathbb C c)} \bar
R\,\mathbf 1$. As in the case of Casimir operator, Sugawara's
relation $q_r$ does not depend on a choice of dual bases
$\{x^i\}_{i\in I}$ and $\{y^i\}_{i\in I}$.
\begin{proposition}\label{P:7.1}
(i) $q_r$ is an element of \ $\ker \Psi$.

(ii) $r\mapsto q_r$ is a $\mathfrak g$-module homomorphism from
$R$ into $\ker \Psi$.

(iii) $x(i)q_r=0$ for all $x\in\mathfrak g$ and $i>0$.
\end{proposition}
Let us denote the set of all Sugawara's relations (\ref{E:7.1}) by
$$
Q_\text{Sugawara}=\{q_r\mid r\in R\}\subset\ker \Psi,
$$
and let us define the $\tilde{\mathfrak g}$-module homomorphism
$$
\Psi_0  \colon U(\tilde{\mathfrak g})\otimes_{U(\tilde{\mathfrak
g}_{\geq 0}+\mathbb C c)} R \to N(k\Lambda_0),\qquad u\otimes w
\mapsto uw.
$$
Then we have:
\begin{proposition}\label{P:7.2} As a $(\mathbf C
L_{-1}\ltimes\tilde{\mathfrak g}_{\leq 0})$-module  $\ker \Psi$ is
generated by
$$
\ker\Psi_0+Q_\text{Sugawara}\,.
$$
\end{proposition}

Let us denote by $\alpha_*$ all simple roots  of $\tilde{\mathfrak
g}$ connected with $\alpha_0$ in a Dynkin diagram:
$$
\alpha_*\neq \alpha_0\,,\qquad
\langle\alpha_0,\alpha_*^\vee\rangle\neq 0.
$$
For $A_n^{(1)}$, $n\geq 2$, there are exactly two such simple
roots, for all the other untwisted affine Lie algebras
$\tilde{\mathfrak g}$ there is exactly one such simple root. In
the case $\tilde{\mathfrak g}\not\cong{\mathfrak sl}(2,\mathbb
C)\,\widetilde{}$\, we have a root vector
$x_{\theta-\alpha_*}=[x_{-\alpha_*},x_{\theta}]$ in the
corresponding finite-dimensional ${\mathfrak g}$.

Since $R$ generates the maximal $\tilde{\mathfrak g}$-submodule
$N^1(k\Lambda_0)$ of $N(k\Lambda_0)$, we have the exact sequence
of $\tilde{\mathfrak g}$-modules
$$
U(\tilde{\mathfrak g})\otimes_{U(\tilde{\mathfrak g}_{\geq
0}+\mathbb C c)} R \xrightarrow{\Psi_0} N(k\Lambda_0)\rightarrow
L(k\Lambda_0)\rightarrow 0.
$$
Generators of $\ker\Psi_0$ can be determined by using
Garland-Lepowsky's resolution
$$
\dots \rightarrow E_2 \rightarrow E_1 \rightarrow E_0\rightarrow
L(k\Lambda_0)\rightarrow 0
$$
of a standard module in terms of generalized Verma modules
\cite{GL}, or by using the BGG type resolution of a standard
module in terms of Verma modules, due to A.~Rocha-Caridi and
N.~R.~Wallach \cite{RW}:
\begin{proposition}\label{P:7.3}
Let $\tilde{\mathfrak g}\not\cong{\mathfrak sl}(2,\mathbb
C)\,\widetilde{}$\,  be an untwisted affine Lie algebra. Then
$\ker\Psi_0$ is generated by the singular vector(s)
$$
x_{\theta-\alpha_*}(-1)\otimes x_\theta(-1)^{k+1}\mathbf 1\mathbf
-x_{\theta}(-1)\otimes
x_{\theta-\alpha_*}(-1)x_\theta(-1)^{k}\mathbf 1,\quad
\langle\alpha_0,\alpha_*^\vee\rangle\neq 0.
$$
\end{proposition}
\bigskip

By combining Theorem~\ref{T:6.1} and Propositions~\ref{P:7.2} and
\ref{P:7.3} we have a description of generators of relations for
annihilating fields:
\begin{theorem}\label{T:7.4}
Let $\tilde{\mathfrak g}\not\cong{\mathfrak sl}(2,\mathbb
C)\,\widetilde{}$\,  be an untwisted affine Lie algebra. Then the
$(\mathbf C L_{-1}\ltimes\tilde{\mathfrak g}_{\leq 0})$-module
$\ker \Phi$ is generated by vectors
$$
 x_\theta(-1)^{k+1}\mathbf 1 \otimes x_{\theta-\alpha_*}(-1)\mathbf 1
- x_{\theta-\alpha_*}(-1)x_\theta(-1)^{k}\mathbf 1\otimes
x_{\theta}(-1)\mathbf 1,\quad
\langle\alpha_0,\alpha_*^\vee\rangle\neq 0,
$$
$$
\tfrac{1}{k+g^\vee}\sum_{i\in I}y^i(0)x_\theta(-1)^{k+1}\mathbf 1
\otimes x^i(-1)\mathbf 1 + L_{-1}\big(\tfrac{1}{k+g^\vee}\,\Omega
-1\big)x_\theta(-1)^{k+1}\mathbf 1 \otimes\mathbf 1.
$$
\end{theorem}

This description of generators of relations for annihilating
fields has some disadvantages when it comes to combinatorial
applications. Namely, the obvious relation
$$
x_\theta(z)^{k+1}\tfrac{d}{dz}x_\theta(z)-
\tfrac{1}{k+1}\,\tfrac{d}{dz}(x_\theta(z)^{k+1}) x_\theta(z)=0
$$
for the annihilating field $x_\theta(z)^{k+1}$ comes from the
element
\begin{equation}\label{E:7.2}
q_{(k+2)\theta}=x_{\theta}(-2)\otimes x_\theta(-1)^{k+1}\mathbf
1\mathbf -x_{\theta}(-1)\otimes
x_{\theta}(-2)x_\theta(-1)^{k}\mathbf 1
\end{equation}
in $\ker\Psi$. This element $q_{(k+2)\theta}$ has length $k+2$ in
the natural filtration, but when written in terms of generators
described in Theorem~\ref{T:7.4}, it is expressed in terms of
elements of length $>k+2$. On the other hand, we can obtain from
(\ref{E:7.2}) both the singular vector(s)
$$
q_{(k+2)\theta-\alpha_*}=x_{\theta-\alpha_*}(-1)\otimes
x_\theta(-1)^{k+1}\mathbf 1\mathbf -x_{\theta}(-1)\otimes
x_{\theta-\alpha_*}(-1)x_\theta(-1)^{k}\mathbf 1
$$
in $\ker\Psi_0$ and the Sugawara singular vector
$$
q_{(k+1)\theta}=\tfrac{1}{k+g^\vee}\sum_{i\in I}x^i(-1)\otimes
y^i(0)x_\theta(-1)^{k+1}\mathbf 1-1\otimes
Dx_\theta(-1)^{k+1}\mathbf 1
$$
by using the action of $\tilde{\mathfrak g}_{\geq 0}$ on
$\ker\Psi$:
\begin{lemma}\label{L:7.5} Let $\Omega$ be the Casimir
operator for ${\mathfrak g}\not\cong{\mathfrak sl}(2,\mathbb C)$
and $\lambda=(k+2)\theta-\alpha_*$. Then
\begin{align*}
&q_{(k+2)\theta-\alpha_*}=x_{-\alpha_*}(1)q_{(k+2)\theta},\\
&q_{(k+1)\theta}=\tfrac{k+1}{2(k+2)(k+g^\vee)}\big(\Omega-(\lambda+2\rho,\lambda)\big)
x_{-\theta}(1)q_{(k+2)\theta}.
\end{align*}
\end{lemma}
For any untwisted affine Lie algebra $\tilde{\mathfrak g}$,
including ${\mathfrak sl}(2,\mathbb C)\,\widetilde{}$\,, the
$(\mathbf C L_{-1}\ltimes\tilde{\mathfrak g})$-module $\ker \Psi$
is generated by the vector $q_{(k+2)\theta}$. This generator plays
an important role in combinatorial applications.

\section{Leading terms}

The associative algebra $\mathcal U=U(\hat{\mathfrak g})/(c-k)$
inherits {}from $U(\hat{\mathfrak g})$ the filtration $\mathcal
U_\ell$, $\ell\in\mathbb Z_{\geq 0}$; let us denote by $\mathcal
S\cong S(\bar{\mathfrak g})$ the corresponding commutative graded
algebra.

Let $B$ be a basis of ${\mathfrak g}$. We fix the basis $\tilde B$
of $\tilde{\mathfrak g}$,
$$
\tilde B=\bar B\cup\{c, d\},\quad \bar B=\bigcup_{j\in\mathbb Z}
B\otimes t^j,
$$
so that $\bar B$ may also be viewed as a basis of $\bar{\mathfrak
g}=\hat{\mathfrak g}/\mathbb C c$. Let $\preceq$ be a linear order
on $\bar B$ such that
$$
i<j\quad\text{implies}\quad x(i)\prec y(j).
$$

 The symmetric algebra $\mathcal S$ has a basis $\mathcal P$
consisting of monomials in basis elements $\bar B$. Elements
$\pi\in\mathcal P$ are finite products of the form
$$
\pi=\prod_{i=1}^\ell b_i(j_i), \quad b_i(j_i)\in\bar B,
$$
and we shall say that $\pi$ is a colored partition of degree
$|\pi|=\sum_{i=1}^\ell j_i\in \mathbb Z$ and length $\ell
(\pi)=\ell$, with parts $b_i(j_i)$ of degree $j_i$ and color
$b_i$. We shall usually assume that parts of $\pi$ are indexed so
that
$$
b_1(j_1)\preceq b_2(j_2)\preceq \dots\preceq b_\ell(j_\ell).
$$
We associate with a colored partition $\pi$  its shape
$\text{sh\,}\pi$, the ``plain'' partition
$$
j_1\leq j_2\leq \dots\leq j_\ell.
$$
The basis element $1\in\mathcal P$ we call the colored partition
of degree 0 and length 0, we may also denote it by $\varnothing$,
suggesting it has no parts. The set of all colored partitions of
degree $n$ and length $\ell$ is denoted as $\mathcal P^\ell(n)$.
The set of all colored partitions with parts $b_i(j_i)$ of degree
$j_i<0$ (respectively $j_i\leq 0$) is denoted as $\mathcal P_{<0}$
(respectively $\mathcal P_{\leq 0}$).

Note that $\mathcal P\subset\mathcal S$ is a monoid with the unit
element 1, the product of monomials $\kappa$ and $\rho$ is denoted
by $\kappa\rho$. For colored partitions $\kappa$, $\rho$ and
$\pi=\kappa\rho$ we shall write $\kappa=\pi/\rho$ and
$\rho\subset\pi$. We shall say that $\rho\subset\pi$ is an
embedding (of $\rho$ in $\pi$), notation suggesting that $\pi$
``contains'' all the parts of $\rho$.

We shall fix a monomial basis
$$
u(\pi)=b_1(j_1)b_2(j_2) \dots b_n(j_\ell), \quad \pi\in\mathcal P,
$$
of the enveloping algebra $\mathcal U$.

Clearly $\bar B\subset\mathcal P$, viewed as colored partitions of
length 1. We assume that on $\mathcal P$ we have a linear order
$\preceq$ which extends the order $\preceq$ on $\bar B$. Moreover,
we assume that order $\preceq$ on $\mathcal P$ has the following
properties:
\begin{itemize}
\item  $\ell(\pi)>\ell(\kappa)$ implies $\pi\prec\kappa$.
\item  $\ell(\pi)=\ell(\kappa)$, $|\pi|<|\kappa|$ implies $\pi\prec\kappa$.
\item  Let $\ell(\pi)=\ell(\kappa)$, $|\pi|=|\kappa|$. Let $\pi$ be a partition
$b_1(j_1)\preceq b_2(j_2)\preceq \dots\preceq b_\ell(j_\ell)$ and
$\kappa$ a partition $a_1(i_1)\preceq a_2(i_2)\preceq \dots\preceq
a_\ell(i_\ell)$. Then $\pi\preceq \kappa$ implies $j_\ell\leq
i_\ell$.
\item  Let $\ell \ge 0$, $n \in  \mathbb Z$ and let
$S \subset \mathcal P$ be a nonempty subset such that all $\pi$ in
$S$ have length $\ell(\pi) \le \ell$ and degree $\vert \pi \vert =
n$. Then $S$ has a minimal element.
\item   $\mu  \preceq  \nu$ implies
$\pi\mu  \preceq \pi\nu$.
\item  The relation $\pi\prec\kappa$ is a well order on $\mathcal P_{\leq 0}$.
\end{itemize}

\begin{remark}
An order with these properties is used in \cite{MP1}; colored
partitions are compared first by length and degree, and then by
comparing degrees of parts and colors of parts in the reverse
lexicographical order.
\end{remark}

For $\pi\in \mathcal P$, $|\pi|=n$, set

$$
U_{[\pi]}^\mathcal P =\overline{\mathbb C\text{-span}\{u(\pi')\mid
|\pi'|=|\pi|, \pi' \succeq\pi\}}\ ,
$$

$$
U_{(\pi)}^\mathcal P  =\overline{\mathbb
C\text{-span}\{u(\pi')\mid |\pi'|=|\pi|, \pi' \succ \pi\}}\ ,
$$

\noindent
 the closure taken in $\overline{\mathcal U(n)}$. Set
$$
U^\mathcal P (n)=  \bigcup_{\pi\in \mathcal P, \ |\pi|=n}
U_{[\pi]}^\mathcal P ,\qquad U^\mathcal P =\coprod_{n\in\mathbb
Z}U^\mathcal P (n) \subset\overline{\mathcal U}.
$$
The construction of $U^\mathcal P$ depends on a choice of
$(\mathcal P, \preceq)$. Since by assumption $\mu  \preceq  \nu$
implies $\pi\mu  \preceq \pi\nu$, we have that $U^\mathcal P $ is
a subalgebra of $\overline{\mathcal U}$. Moreover, we have a
sequence of subalgebras:
\begin{proposition}
\label{P:8.1} \quad$\mathcal U\subset U \subset U^\mathcal P
\subset \overline{\mathcal U}$.
\end{proposition}

As in \cite{MP1}, we have:
\begin{lemma}
\label{L:8.2} For $\pi \in \mathcal P $ we have
$U_{[\pi]}^\mathcal P =\mathbb C u(\pi)+ U_{(\pi)}^\mathcal P $.
Moreover,
$$
\dim U_{[\pi]}^\mathcal P / U_{(\pi)}^\mathcal P = 1.
$$
\end{lemma}
\noindent For $u \in U_{[\pi]}^\mathcal P $, $u \notin
U_{(\pi)}^\mathcal P $ we define the {\it leading term}
$$
\ell \!\text{{\it t\,}}(u) = \pi.
$$

\begin{proposition}
\label{P:8.3}  Every element
 $u \in U^\mathcal P (n)$, $u \neq 0$, has a unique leading
 term $\ell \!\text{{\it t\,}}(u)$.
\end{proposition}
By Proposition~\ref{P:8.3} every nonzero homogeneous $u$ has the
unique leading term. For a nonzero element $u\in U^\mathcal P $ we
define the leading term $\ell \!\text{{\it t\,}}(u)$ as the
leading term of the nonzero homogeneous component of $u$ of
smallest degree. For a subset $S\subset U^\mathcal P $ set
$$
\ell \!\text{{\it t\,}}(S)=\{\ell \!\text{{\it t\,}}(u)\mid u\in
S, u\neq 0\}.
$$
We are interested mainly in leading terms of elements in $U
\subset U^\mathcal P$, which have the following properties:
\begin{proposition}
\label{P:8.4} For all $u, v \in U^\mathcal P \backslash\{0\}$ we
have $\ell \!\text{{\it t\,}}(uv)=\ell \!\text{{\it t\,}}(u)\ell
\!\text{{\it t\,}}(v)$.
\end{proposition}
\begin{proposition}
\label{P:8.5} Let $W\subset U^\mathcal P $ be a finite-dimensional
subspace and let \newline $\ell \!\text{{\it t\,}}(W)\rightarrow
W$ be a map such that
$$
\rho\mapsto w(\rho),\qquad \ell \!\text{{\it t\,}}\big(
w(\rho)\big)=\rho.
$$
Then $\{w(\rho)\mid \rho\in\ell \!\text{{\it t\,}}(W)\}$ is a
basis of $W$.
\end{proposition}

Since $R$ is finite-dimensional, the space $\bar R\subset U$ is a
direct sum of finite-dimension\-al homogeneous subspaces. Hence
Proposition~\ref{P:8.5} implies that we can parametrize a basis of
$\bar R$ by the set of leading terms $\ell \!\text{{\it t\,}}
(\bar R)$: we fix a map
$$
\ell \!\text{{\it t\,}} (\bar{R}) \rightarrow \bar R ,\quad\rho
\mapsto r(\rho)\quad \text{such that}\quad r(\rho)\in U(\vert \rho
\vert),\ \ell \!\text{{\it t\,}} (r(\rho)) = \rho,
$$
then $\{r(\rho)\mid \rho\in \ell \!\text{{\it t\,}} (\bar R)\}$ is
a basis of $\bar R$. We will assume that this map is such that the
coefficient $C$ of ``the leading term'' $u(\rho)$ in ``the
expansion'' of $r(\rho)=Cu(\rho)+\dots$ is chosen to be $C=1$.
Note that our assumption $R\subset N^1(k\Lambda_0)$ implies that
$1\not\in \ell \!\text{{\it t\,}} (\bar R)$ and that $ \ell
\!\text{{\it t\,}} (\bar R)\cdot\mathcal P$ is a proper ideal in
the monoid $\mathcal P$.

For an embedding $\rho\subset\pi$, where $\rho\in\ell \!\text{{\it
t\,}} (\bar R)$, we define the element $u(\rho\subset\pi)$ in $ U$
by
$$
u(\rho\subset\pi)=  u(\pi/\rho)r(\rho).
$$

\section{A rank theorem}

Let $a\in V$ be a homogeneous element. Then we have
$$
Y(a,z)=\sum_{n\in\mathbb Z}a(n)z^{-n-\text{wt}\, a}\,,\qquad
a(n)\in U(n).
$$
If $M$ is a level $k$ highest weight $\tilde{\mathfrak g}$-module,
then the action of coefficients $a(n)$ on $M$ makes $M$ a
$V$-module with vertex operators
$$
Y_M(a,z)=\sum_{n\in\mathbb Z}a(n)z^{-n-\text{wt}\, a}\,,\qquad
a(n)\in \text{End}\,M.
$$
Then $M\otimes M$ is a $V\otimes V$-module. For a homogeneous
element $q=a\otimes b$ the vertex operator is defined by
$$
Y_{M\otimes M}(q,z)= Y_M(a,z)\otimes Y_M(b,z)=\sum_{n\in\mathbb
Z}\Big(\sum_{i+j=n}a(i)\otimes b(j)\Big)z^{-n-\text{wt}\,
a-\text{wt}\, a}\,.
$$
Since the condition (\ref{E:2.7}) is satisfied, the coefficient
$$
q(n)=\sum_{i+j=n}a(i)\otimes b(j)
$$
is a well defined operator on $M\otimes M$. On the other hand, we
want to make sense of this formula for $a(i), b(j)\in U$, where
the condition (\ref{E:2.7}) is replaced by the convergence in the
completion $\overline{\mathcal U}$. For this reason set
$$
\big(U\bar\otimes U\big)(n)= \prod_{i+j=n} \big(U(i)\otimes
U(j)\big), \qquad U\bar\otimes U= \coprod_{n\in\mathbb Z}
\big(U\bar\otimes U\big)(n).
$$
The elements of $U\bar\otimes U$ are finite sums of homogeneous
sequences in $U\otimes U$, we shall denote them as
$\sum_{i+j=n}a(i)\otimes b(j)$. For a fixed $n\in\mathbb Z$ we
have a linear map
$$
\chi(n) \colon V\otimes V\rightarrow \big(U\bar\otimes U\big)(n)
$$
defined for homogeneous elements $a$ and $b$ by
$$
\chi(n) \colon a\otimes b\mapsto \sum_{p+r=n}a(p)\otimes b(r).
$$
We think of $\chi(n)(q)$ as ``the coefficient $q(n)$ of the vertex
operator $Y(q,z)$''. We shall write $q(n)=\chi(n)(q)$ for an
element $q\in V\otimes V$ and $Q(n)=\chi(n)(Q)$ for a subspace
$Q\subset V\otimes V$.

Since we have the adjoint action of $\hat{\mathfrak g}$ on $U$, we
define ``the adjoint action'' of $\hat{\mathfrak g}$ on
$U\bar\otimes U$ by
$$
[x(m),\sum_{p+r=n}a(p)\otimes b(r)]=
\sum_{p+r=n}[x(m),a(p)]\otimes b(r)+\sum_{p+r=n}a(p)\otimes
[x(m),b(r)].
$$
Note that we have the action of $\hat{\mathfrak g}$ on $V\otimes
V$ given by
$$
x_i(a\otimes b)= (x_i a)\otimes b+a\otimes (x_i b)\,,\qquad
x\in\mathfrak g, \ i\in\mathbb Z.
$$
As expected, we have the following commutator formula for
$q(n)=\chi(n)(q)$:
\begin{proposition}
\label{P:9.1} For\ $x(m)\in\hat{\mathfrak g}$ and homogeneous
$q\in V\otimes V$ we have
$$
[x(m),q(n)]=\sum_{i\geq 0}\binom{m}{i}(x_iq)(m+n),\qquad
(Dq)(n)=-(n+\text{wt}\, q)q(n).
$$
\end{proposition}
So if a subspace $Q\subset V\otimes V$ is invariant for
$\tilde{\mathfrak g}_{\geq 0}$, then
$$
\coprod_{n\in\mathbb Z} Q(n)
$$
is a loop $\hat{\mathfrak g}$-module, in general reducible.

Now assume that $q=\sum a\otimes b$ is a homogeneous element in
$\bar R{\mathbf 1}\otimes V$. Note that for $a\in\bar R{\mathbf
1}$ the coefficient $a(i)$ of the corresponding field $Y(a,z)$ can
be written as a finite linear combination of basis elements
$r(\rho)$, $\rho\in\ell \!\text{{\it t\,}}(\bar R)$. Hence each
element of the sequence $q(n)=\chi(n)(\sum a\otimes b)\in
\big(U\bar\otimes U\big)(n)$, say $c_i$, can be written uniquely
as a finite sum of the form
$$
c_i=\sum_{\substack{ \rho\in\ell \!\text{{\it t\,}}(\bar R)}}
r(\rho)\otimes b_\rho,
$$
where $b_\rho\in U$. If $b_\rho\neq 0$, then it is clear that
$|\rho|+|\ell \!\text{{\it t\,}}(b_\rho)|=n$. Let us assume that
$q(n)\neq 0$, and for nonzero ``$i$-th'' component $c_i$ let
$\pi_i$ be the smallest possible $\rho\,\ell \!\text{{\it
t\,}}(b_\rho)$ that appears in the expression for $c_i$. Denote by
$S$ the set of all such $\pi_i$. Since $q$ is a finite sum of
elements of the form $a\otimes b$, it is clear that there is
$\ell$ such that $\ell(\pi_i)\leq\ell$. Then, by our assumptions
on the order $\preceq$, the set $S$ has the minimal element, and
we call it the leading term $\ell \!\text{{\it
t\,}}\big(q(n)\big)$ of $q(n)$. For a subspace $Q\subset \bar
R{\mathbf 1}\otimes V$ set
$$
\ell \!\text{{\it t\,}}(Q(n))=\{\ell \!\text{{\it t\,}}(q(n))\mid
q\in Q,\, q(n)\neq 0\}.
$$

For a colored partition $\pi$ of set
$$
N(\pi)=\max \{\#\mathcal E(\pi)\!-\!1,\,0\},\quad \mathcal
E(\pi)=\{\rho\in \ell \!\text{{\it t\,}} (\bar R) \mid
\rho\subset\pi\}.
$$

Note that $V\otimes V$ has a natural filtration $(V\otimes
V)_\ell$, $\ell\in\mathbb Z_{\geq 0}$, inherited {}from the
filtration $\mathcal U_{\,\ell}$, $\ell\in\mathbb Z_{\geq 0}$.
Then we have the following ``rank theorem'':

\begin{theorem}
\label{T:9.2} Let $Q\subset \ker \left(\Phi\, | \,(\bar R \mathbf
1\otimes V)_\ell\right)$
 be a finite-dimensional subspace and $n\in\mathbb Z$.
Assume that $\ell(\pi)=\ell$ for all $\pi\in\ell \!\text{{\it
t\,}} (Q(n))$. If
\begin{equation}\label{E:9.1}
\sum_{\pi\in\mathcal P^\ell(n)} N(\pi)=\dim Q(n),
\end{equation}
then for any two embeddings
 $\rho_1 \subset \pi$ and $\rho_2 \subset \pi$ in $\pi\in\mathcal P^\ell(n)$,
where $\rho_1, \rho_2 \in\ell \!\text{{\it t\,}}(\bar{R})$, we
have a relation
\begin{equation}\label{E:9.2}
u(\rho_1 \subset \pi) \in u(\rho_2 \subset \pi) + \overline
{\mathbb C\text{-span}\{u(\rho \subset \pi')\mid
 \rho \in \ell \!\text{{\it t\,}} (\bar{R}), \rho \subset \pi', \pi \prec \pi'\}}.
\end{equation}
\end{theorem}

Combinatorial relations (\ref{E:9.2}) for the defining relations
$r(\rho)$ of standard modules are needed for construction of
combinatorial bases of standard modules. The left hand side of
(\ref{E:9.1}) is, for a given degree $n$, the total number $N(n)$
of relations needed, and the right hand side of (\ref{E:9.1}) is
the number of relations that we can construct by using the
representation theory. As expected, $N(n)\geq \dim Q(n)$.

It should be noted that relations of the form (\ref{E:9.2}) are
easy to obtain when $\rho_1\rho_2 \subset \pi$. The problem is
when two embeddings ``intersect''. Such relations for $r(\rho)$ of
the combinatorial form (\ref{E:9.2}) are obtained as linear
combinations of relations constructed from ``coefficients $q(n)$
of vertex operators $Y(q,z)$''. In another words, a relation of
the form (\ref{E:9.2}) is a solution of certain system of linear
equations, its existence is guaranteed by the condition
(\ref{E:9.1}).

\section{The problem of constructing a combinatorial basis of $L(k\Lambda_0)$}

We shall illustrate the (desired) construction of combinatorial
bases of standard modules on the simpler case of $L(k\Lambda_0)$.

We assume we have an ordered basis $B$ and we define the order
$\preceq$ on $\mathcal P$ by comparing partitions gradually
\begin{enumerate}
\item  by length,
\item  by degree,
\item  by shape with reverse lexicographical order,
\item  by colors with reverse lexicographical order.
\end{enumerate}

Set $r_{(k+1)\theta}=x_\theta(-1)^{k+1}\mathbf 1$. Then, as in
\cite{MP1}, we have
$$
\ell \!\text{{\it
t\,}}\left(r_{(k+1)\theta}(n)\right)=x_\theta(-j-1)^ax_\theta(-j)^b
$$
with $a+b=k+1$ and $(-j-1)a+(-j)b=n$. Since we can obtain all
other elements $r(n)$ for $r\in R$ by the adjoint action of
$\mathfrak g$, which does not change the length and degree, we
have that shapes of leading terms of $r(n)$ remain the same:
\begin{equation}\label{shapes}
\text{sh\,}\ell \!\text{{\it t\,}}\big(r(n)\big)=(-j-1)^a(-j)^b
\end{equation}
with $a+b=k+1$ and $(-j-1)a+(-j)b=n$. Let us introduce the
notation
$$
\mathcal D=\ell \!\text{{\it t\,}}(\bar{R})\cap\mathcal
P_{<0}\,,\qquad \mathcal{RR}=P_{<0}\backslash \big(\mathcal D\cdot
P_{<0}\big).
$$

We shall denote by $\mathbf 1$ the highest weight vector in the
standard module $L(k\Lambda_0)=N(k\Lambda_0)/N^1(k\Lambda_0)$.

\begin{proposition}
\label{P:10.1} If for each $\ell\in\{k+2,\dots,2k+1\}$ there
exists a finite-dimen\-sion\-al subspace
$Q_\ell\subset\ker\left(\Phi\, | \,(\bar R \mathbf 1\otimes
V)_\ell\right)$ such that $\ell(\pi)=\ell$ for all $\pi\in\ell
\!\text{{\it t\,}} (Q_\ell(n))$ and
$$
\sum_{\pi\in\mathcal P^\ell(n)} N(\pi)=\dim Q_\ell(n),
$$
for all $n\leq -k-2$, then the set of vectors
\begin{equation}\label{E:10.3}
u(\pi)\mathbf 1\,, \qquad \pi\in\mathcal{RR},
\end{equation}
is a basis of the standard module $L(k\Lambda_0)$.
\end{proposition}
\begin{proof}
Since elements in $R$ are of degree $k+1$, and there is no element
in  $N^1(k\Lambda_0)$ of smaller degree, for $\rho \in \ell
\!\text{{\it t\,}} (\bar{R})$ we have that $r(\rho)\mathbf 1=0$
whenever $|\rho|>-k-1$. Hence (\ref{shapes}) implies that $\rho\in
\mathcal D$ whenever $r(\rho)\mathbf 1\neq 0$. Since
$N^1(k\Lambda_0)=\bar R N(k\Lambda_0)=U(\tilde{\mathfrak
g}_{<0})\bar R\mathbf 1$, we have a spanning set of
$N^1(k\Lambda_0)$
$$
u(\kappa)r(\rho)\mathbf 1=u(\rho\subset\kappa\rho)\mathbf 1\,,
\qquad \kappa\in\mathcal P_{<0}\,, \ \rho\in\mathcal D.
$$
For each $\pi\in\mathcal D\cdot P_{<0}$ choose exactly one
$\rho_\pi\in\mathcal D$ such that $\rho_\pi\subset\pi$. Since by
our assumptions we can apply Theorem~\ref{T:9.2}, for each
$\pi\in\mathcal D\cdot P_{<0}$ such that
$\pi=\kappa_1\rho_1=\kappa_2\rho_2$ we have a relation
(\ref{E:9.2}). Hence, by using induction, we se that
\begin{equation}\label{E:10.4}
u(\rho_\pi\subset\pi)\mathbf 1\,, \qquad \pi\in\mathcal D\cdot
P_{<0}\,,
\end{equation}
is a spanning set of $N^1(k\Lambda_0)$. Since by
Proposition~\ref{P:8.4}
$$
\ell \!\text{{\it t\,}}\big(u(\pi/\rho_\pi)r(\rho_\pi)\big)
=\big(\pi/\rho_\pi)\cdot\rho_\pi=\pi\,,
$$
we have that
$$
u(\rho_\pi\subset\pi)\mathbf 1\in u(\pi)\mathbf 1 +
U_{(\pi)}^\mathcal P \,\mathbf 1,
$$
and by induction we see that the set (\ref{E:10.4}) is linearly
independent. Hence this set is a basis of $N^1(k\Lambda_0)$.

In the obvious way we can assign to each colored partition $\pi$
its weight $\text{wt\,}\pi$, and we have characters
$$
\text{ch\,}N(k\Lambda_0)=\sum_{\pi\in\mathcal
P_{<0}}e^{\text{wt\,}\pi},\qquad
\text{ch\,}N^1(k\Lambda_0)=\sum_{\pi\in\mathcal D\cdot
P_{<0}}e^{\text{wt\,}\pi}.
$$
Hence we have
\begin{equation}\label{E:10.5}
\text{ch\,}L(k\Lambda_0)=\sum_{\pi\in\mathcal{RR}}e^{\text{wt\,}\pi}.
\end{equation}

To find a basis of $L(k\Lambda_0)$ we start with the PBW spanning
set
$$
u(\pi)\mathbf 1\,, \qquad \pi\in\mathcal P_{<0}\,.
$$
For $\pi\in\mathcal D\cdot P_{<0}$ we have
$$
u(\pi)\in u(\pi/\rho_\pi)r(\rho_\pi)+U_{(\pi)}^\mathcal P .
$$
Since $r(\rho_\pi)\mathbf 1=0$ in $L(k\Lambda_0)$, we have
$$
u(\pi)\mathbf 1\in U_{(\pi)}^\mathcal P\,\mathbf 1 \quad\text{for}
\ \pi\in\mathcal D\cdot P_{<0},
$$
and by using induction we can reduce the PBW spanning set to a
spanning set (\ref{E:10.3}). By the character formula
(\ref{E:10.5}) this set is linearly independent.
\end{proof}
\begin{remarks}
(i) At the moment just a few examples are known where the
conditions of Theorem~\ref{T:9.2} are satisfied, the simplest is
for the basic ${\mathfrak sl}(2,\mathbb C)\,\widetilde{}\,$-module
(see \cite{MP1}). With the usual notation $x=x_{\theta}$,
$h=\theta^\vee$ and $y=x_{-\theta}$, the set of leading terms
$\ell \!\text{{\it t\,}}(\bar{R})$ is:
$$
b_1(-j)b_2(-j)\quad\text{with colors \ } b_1b_2:\quad yy, yh, hh,
hx, xx,
$$
$$
b_1(-j-1)b_2(-j)\quad\text{with colors \ } b_1b_2:\quad yy, hy,
xy, xh, xx.
$$
If one takes
$$
Q_3=U(\mathfrak g)q_{2\theta}\oplus U(\mathfrak g)q_{3\theta},
$$
then, by using Proposition~\ref{P:9.1} and loop modules, $\dim
Q_3(n)=5+7$ and (\ref{E:9.1}) holds for all $n$. If one takes
$(1,2)$-specialization of the Weyl-Kac character formula on one
side, and (\ref{E:10.5}) on the other side, one obtains a
Capparelli identity \cite{C}.

(ii) The results in Section 9 can be extended to twisted
affine Lie algebras (see \cite{P1}). In such formulation of
Theorem~\ref{T:9.2} the equality (\ref{E:9.1}) also holds for
level $1$ twisted ${\mathfrak sl}(3,\mathbb
C)\,\widetilde{}\,$-modules (see \cite{S}).

(iii) The character formula (\ref{E:10.5}) is a generating
function for numbers of colored partitions in $\mathcal{RR}$
satisfying ``difference $\mathcal{D}$ conditions'', and combined
with the Weyl-Kac character formula gives a Rogers-Ramanujan type
identity.
\end{remarks}
\section{Combinatorial bases of basic modules for $C_n\sp{(1)}$}

We fix a simple Lie algebra $\mathfrak{g}$ of type $C_n$, $n\geq 2$. For a given
Cartan subalgebra $\mathfrak h$ and the corresponding
root system $\Delta$ we can write
\begin{equation*}
\Delta = \{\pm(\varepsilon_i\pm\varepsilon_j) \mid i,j=1,...,n\}\setminus\{0\}\ .
\end{equation*}
We chose simple roots as in \cite{Bou}
\begin{equation*}
\alpha_1= \varepsilon_1-\varepsilon_2,  \ \alpha_2=\varepsilon_2-\varepsilon_3, \  \cdots \ \alpha_{n-1}=\varepsilon_{n-1}-\varepsilon_{n},
 \  \alpha_n = 2\varepsilon_n.
\end{equation*}
Then $\theta=2 \varepsilon_1$ and $\alpha^{\star}=\alpha_1$. By Lemma 7.5 for each degree $m$ we have a space of relations for annihilating fields
$$
Q_3(m)=U(\mathfrak g)q_{2\theta}(m)\oplus U(\mathfrak g)q_{3\theta}(m)\oplus U(\mathfrak g)q_{3\theta-\alpha\sp*}(m)\ .
$$
The Weyl dimension formula
for $\mathfrak{g}$ gives
\begin{eqnarray}
\dim L(2\theta) &=& {2n+3\choose 4},\label{E:11.1}\\
\dim L(3\theta) &=& {2n+5\choose 6},\label{E:11.2}\\
\dim L(3\theta-\alpha^{\star}) &=& \frac{(2n+5)(n-1)}{3}{2n+3\choose 4}.\label{E:11.3}
\end{eqnarray}
Hence we have
\begin{equation}\label{tri sume}
\dim Q_3(m) = \dim L(2\theta) + \dim L(3\theta)  + \dim L(3\theta-\alpha^{\star}) =2n
{2n+4\choose 5} .
\end{equation}
For each $\alpha\in\Delta$ we chose a root vector $x_{\alpha}$ such that
$[x_{\alpha},x_{-\alpha}]=\alpha^{\vee}$. For root vectors
$x_{\alpha}$ we shall use the following notation:
$$\begin{array}{ccc}
x_{ij}\quad \text{or just}\quad ij &  \text{if}\   &\alpha =\varepsilon_i + \varepsilon_j\ , \ i\leq j\,,\\
x_{\underline{i}\underline{j}}\quad \text{or just}\quad \underline{i}\underline{j}
 & \ \text{if}\  &\alpha =-\varepsilon_i - \varepsilon_j\ , \ i\geq j\,,\\
x_{i\underline{j}}\quad \text{or just}\quad i \underline{j} & \ \text{if}\
&\alpha =\varepsilon_i - \varepsilon_j\ , \ i\neq j\,.\\
\end{array}
$$
With previous notation $x_\theta=x_{11}$. We also write for $i=1, \dots, n$
$$
x_{i\underline{i}}=\alpha_i^{\vee}\ \text{or just}\ i\underline{i} \,.
$$
These vectors $x_{ab}$ form a basis $B$ of $\mathfrak g$ which we shall write in
a triangular scheme. For example, for $n=3$ the basis $B$ is
$$\begin{array}{cccccc}
11 &  &&  & & \\
12 & 22 & & & & \\
13 & 23 & 33 & & & \\
1\underline{3} & 2\underline{3} & 3\underline{3} & \underline{3}\underline{3}  & & \\
1\underline{2} & 2\underline{2} & 3\underline{2} & \underline{3}\underline{2}  & \underline{2}\underline{2}& \\
1\underline{1} & 2\underline{1} & 3\underline{1} & \underline{3}\underline{1}  & \underline{2}\underline{1} & \underline{1}\underline{1}.
\end{array}
$$
In general for the set of indices $\{1,2,\cdots ,n,\underline{n},\cdots ,\underline{2},\underline{1}\}$ we use order
\begin{equation*}
1\succ 2\succ\cdots\succ n-1\succ n\succ \underline{n} \succ  \underline{n-1} \succ\cdots
\succ \underline{2} \succ \underline{1}
\end{equation*}
and a basis element $x_{ab}$ we write in $a^{th}$ column and $b^{th}$ row,
\begin{equation}\label{E:12.6}
B=\{x_{ab}\mid b\in\{1,2,\cdots ,n,\underline{n},\cdots ,\underline{2},\underline{1}\},\ a\in\{1,\cdots ,b\}\} .
\end{equation}
By using (\ref{E:12.6}) we define on $B$ the corresponding reverse  lexicographical order, i.e.
\begin{equation}\label{orderB}
x_{ab}\succ x_{a'b'} \ if\ b\succ b'\ or \ b=b' \ and\ a\succ a'\ .
\end{equation}
In other words, $x_{ab}$ is larger than $x_{a' b'}$ if $x_{a' b'}$ lies in a row $b'$ below the row $b$, or $x_{ab}$ and
$x_{a' b'}$ are in the same row $b=b'$, but $x_{a'b'}$  is to the right of $x_{ab}$.

For $r\in \{1,\cdots , n,\underline{n}, \cdots , \underline{1}\}$ we introduce the notation
$$
\triangle_r \quad\text{and} \quad {}\sp r\!\triangle
$$
for triangles in $B$ consiting of rows $\{1,\dots,r\}$ and columns $\{r,\dots,\underline{1}\}$.   For example, for $n=3$ and $r=\underline{3}$ we have triangles $\triangle_{\underline{3}}$ and $ {}\sp {\underline{3}}\triangle$
$$\begin{array}{cccccc}
11 &  &&  & & \\
12 & 22 & & & & \\
13 & 23 & 33 & & & \\
1\underline{3} & 2\underline{3} & 3\underline{3} & \underline{3}\underline{3}  & & \\
&& & \underline{3}\underline{2}  & \underline{2}\underline{2}& \\
&&& \underline{3}\underline{1}  & \underline{2}\underline{1} & \underline{1}\underline{1}.
\end{array}
$$
With order $\preceq$ on $B$ we  define a linear order on $\bar B=\{x(j)\mid x\in B, j\in\mathbb Z\}$ by
\begin{equation}\label{orderB(j)}
x_\alpha(i)\prec x_\beta(j)\quad\text{if}\quad i<j \quad\text{or}\quad i=j, \ x_\alpha\prec x_\beta.
\end{equation}
With order $\preceq$ on $\bar B$ we  define a linear order on  $\mathcal{P}$ by
$$
\pi\prec\kappa\quad\text{if}
$$
\begin{itemize}
\item  $\ell(\pi)>\ell(\kappa)$ or
\item  $\ell(\pi)=\ell(\kappa)$, $|\pi|<|\kappa|$ or
\item  $\ell(\pi)=\ell(\kappa)$, $|\pi|=|\kappa|$, $\text{sh\,}\pi\prec\text{sh\,}\kappa$ in the  reverse
lexicographical order or
\item  $\ell(\pi)=\ell(\kappa)$, $|\pi|=|\kappa|$, $\text{sh\,}\pi=\text{sh\,}\kappa$ and colors of
$\pi$ are smaller than the colors of $\kappa$ in reverse lexicographical order.
\end{itemize}
\begin{lemma}\label{L: leading terms}
The set of leading terms of relations $\bar R$ for level $1$ standard $\tilde{\mathfrak g}$-modules
consists of quadratic monomials
$$
x_{a_1b_1}(-j) x_{a_2b_2}(-j), \quad j\in\mathbb Z, \quad b_{2}\preceq b_1 \ \text{and} \ a_{2}\succeq a_1,
$$
and quadratic monomials
$$
x_{a_1b_1}(-j-1) x_{a_2b_2}(-j),\quad j\in\mathbb Z, \quad   b_{1} \succeq a_2.
$$
\end{lemma}
This lemma is a special case of Theorem 6.1 in \cite{PS}. The proof for this level one case reduces to a very simple argument.
\begin{remark}
Note that a quadratic monomial $x_{a_1b_1}(-j-1) x_{a_2b_2}(-j)$ is a leading term of relation if and only if
there is $r$ such that
\begin{equation*}
x_{a_1b_1}\in\triangle_r \quad\text{and}\quad x_{a_2b_2}\in {}\sp r\!\triangle.
\end{equation*}
\end{remark}
\begin{theorem}\label{T: a basis of basic module}
The set of monomial vectors which have no leading term as a factor,
i.e., the set of vectors
\begin{equation}\
u(\pi)\mathbf 1\,, \qquad \pi\in\mathcal{RR},
\end{equation}
is a basis of the basic $\tilde{\mathfrak g}$-module $L(\Lambda_0)$.
\end{theorem}
\begin{proof}
By Proposition \ref{P:10.1} and  (\ref{tri sume}) it is enough to show
\begin{equation}\label{sum P^3(n)}
\sum_{\pi\in\mathcal P^3(m)} N(\pi)  = 2n {2n+4\choose 5}\ .
\end{equation}
In order to simplify the counting of embeddings of leading terms we introduce a slightly different indexation of a triangular scheme for a basis $B$.  By using
\begin{equation}\label{reindexation}
k \mapsto  k \quad \underline{k} \mapsto  2n-k+1
\end{equation}
and matrix notation for rows and columns we can rewrite the basis $$B=\{x_{k,l} \mid k\in\{1,\cdots, 2n\}\ ,\ l\in\{1,\cdots , k\} \}\ .$$
We need to count embeddings in (\ref{sum P^3(n)}) for $m=-3j-1$, $-3j-2$ and $-3j-3$, that is, we need to
consider three cases:
\begin{itemize}
\item[(I)]$x_{k_1l_1}(-j-1)x_{k_2l_2}(-j)x_{k_3l_3}(-j)$ where $x_{k_2l_2}\preceq x_{k_3l_3}$
\item[(II)]$x_{k_1l_1}(-j-1)x_{k_2l_2}(-j-1)x_{k_3l_3}(-j)$ where $x_{k_1l_1}\preceq x_{k_2l_2}$
\item[(IIIa)]$x_{k_1l_1}(-j-2)x_{k_2l_2}(-j-1)x_{k_3l_3}(-j)$ 
\item[(IIIb)] $x_{k_1l_1}(-j-1)x_{k_2l_2}(-j-1)x_{k_3l_3}(-j-1)$ where $x_{k_1l_1}\preceq x_{k_2l_2}\preceq x_{k_3l_3}$\ .
\end{itemize}
Denote by $N$ the number of embeddings. During counting embeddings of leading terms we need to multiply the count by a factor $N-1$.
We describe calculation of the first case in all details.
\medskip\\
\emph{The first case} $x_{k_1l_1}(-j-1)x_{k_2l_2}(-j)x_{k_3l_3}(-j)$ where $x_{k_2l_2}\preceq x_{k_3l_3}$.\smallskip\\
Depending on the type and number of embeddings the first case is split in the following five subcases:
\begin{itemize}
\item[(I1)] $N=3$; $x_{k_1l_1}(-j-1)x_{k_2l_2}(-j)$, $x_{k_1l_1}(-j-1)x_{k_3l_3}(-j)$ and $x_{k_2l_2}(-j)x_{k_3l_3}(-j)$ are leading terms  (+ condition $x_{k_2l_2}\neq x_{k_3l_3}$)
\item[(I2)] $N=2$; $x_{k_1l_1}(-j-1)x_{k_2l_2}(-j)$, $x_{k_1l_1}(-j-1)x_{k_3l_3}(-j)$ and $x_{k_2l_2}(-j)x_{k_3l_3}(-j)$ are leading term  (+ condition $x_{k_2l_2} = x_{k_3l_3}$)
\item[(I3)] $N=2$; $x_{k_1l_1}(-j-1)x_{k_2l_2}(-j)$, $x_{k_1l_1}(-j-1)x_{k_3l_3}(-j)$ are leading terms and $x_{k_2l_2}(-j)x_{k_3l_3}(-j)$ not leading terms (+ condition $x_{k_2l_2}\neq x_{k_3l_3}$)
\item[(I4)] $N=2$; $x_{k_1l_1}(-j-1)x_{k_2l_2}(-j)$ not leading term, $x_{k_1l_1}(-j-1)x_{k_3l_3}(-j)$ and $x_{k_2l_2}(-j)x_{k_3l_3}(-j)$ are leading terms (+ condition $x_{k_2l_2}\neq x_{k_3l_3}$)
\item[(I5)] $N=2$; $x_{k_1l_1}(-j-1)x_{k_2l_2}(-j)$ is leading term, $x_{k_1l_1}(-j-1)x_{k_3l_3}(-j)$ not leading term and $x_{k_2l_2}(-j)x_{k_3l_3}(-j)$ is leading term (+ condition $x_{k_2l_2}\neq x_{k_3l_3}$)
\end{itemize}
\emph{Subcase} (I1):
Note that $x_{k_1l_1}(-j-1) \in\triangle_r$ and $x_{k_2l_2}(-j)\in{}\sp r\!\triangle$ and selection of their position is entirely free. Therefore, the  number of embedded leading terms in this subcase is given by
\begin{equation}\label{sumaI1}
\sum_{I1} = \sum_{k_1=1}^{2n}\sum_{l_1=1}^{k_1}\sum_{k_2=k_1}^{2n}\sum_{l_2=k_1}^{k_2} [(N-1)(\sharp \ x_{k_3l_3})]
\end{equation}
where $\sharp \ x_{k_3l_3}$ is number of admissible position for $x_{k_3l_3}$. Since $x_{k_2l_2}\preceq x_{k_3l_3}$ then $\sharp \ x_{k_3l_3} = (l_2-k_1)+[1+2+\cdots +(k_2-l_2)]$ (see Figure I1) and  the  sum (\ref{sumaI1}) is
\begin{equation}\label{sumaI1 final}
\sum_{I1}=\sum_{k_1=1}^{2n}\sum_{l_1=1}^{k_1}\sum_{k_2=k_1}^{2n}\sum_{l_2=k_1}^{k_2} [2(l_2-k_1)+(k_2-l_2)(k_2-l_2+1)]
\end{equation}

$$
\begin{array}{cccccccccccccc}
\cdot &&&&&&&&&&&&& \\
\cdot & \cdot &&&&&&&&&&&& \\
\cdot & \cdot & \cdot &&&&&&&&&&& \\
\cdot & \cdot & \cdot & \cdot  &&&&&&&&&& \\
\cdot & \cdot & \cdot & \cdot  & \cdot &&&&&&&&& \\
\cdot & \cdot & \cdot & \star  & \cdot & \cdot &&&&&&&& \\
&&&&& \cdot &\cdot &&&&&&& \\
&&&&& \cdot & \cdot & \bullet &&&&&& \\
&&&&& \cdot & \cdot & \bullet & \bullet &&&&& \\
&&&&& \cdot & \cdot & \bullet & \bullet & \bullet &&&& \\
&&&&& \cdot & \cdot & \bullet & \bullet & \bullet & \bullet &&& \\
&&&&& \bullet & \bullet & \ast & \cdot & \cdot & \cdot & \cdot && \\
&&&&& \cdot & \cdot & \cdot & \cdot & \cdot & \cdot & \cdot & \cdot & \\
&&&&& \cdot & \cdot & \cdot & \cdot & \cdot & \cdot & \cdot & \cdot & \cdot
\end{array}
$$
$$ Figure\ I1\ (\star \sim x_{k_1l_1};\ \ast\sim x_{k_2l_2};\ \bullet\sim x_{k_3l_3})$$
\emph{Subcase} (I2):
This subcase is similar as subcase (I1) for $N=2$ and $\sharp \ x_{k_3l_3} =1$. From this immediately follows
\begin{equation}\label{sumaI2 final}
\sum_{I2}=\sum_{k_1=1}^{2n}\sum_{l_1=1}^{k_1}\sum_{k_2=k_1}^{2n}\sum_{l_2=k_1}^{k_2}1\ .
\end{equation}
\emph{Subcase} (I3):
In this subcase we have again the same following setting
$$N=2\ ; x_{k_2l_2}\prec x_{k_3l_3}\ ;\ x_{k_1l_1}(-j-1) \in\triangle_r\ ;\ x_{k_2l_2}(-j)\in{}\sp r\!\triangle\ .$$
Since the $x_{k_2l_2}(-j)x_{k_3l_3}(-j)$ is not leading term then $\sharp \ x_{k_3l_3} = \frac{(l_2-k_1)(2k_2-k_1-l_2+1)}{2}$ (see Figure I3) and  the  sum $\sum_{I3}$ is
\begin{equation}\label{sumaI3 final}
\sum_{I3}=\sum_{k_1=1}^{2n}\sum_{l_1=1}^{k_1}\sum_{k_2=k_1}^{2n}\sum_{l_2=k_1}^{k_2} [\frac{(l_2-k_1)(2k_2-k_1-l_2+1)}{2}]\ .
\end{equation}

$$
\begin{array}{cccccccccccccc}
\cdot &&&&&&&&&&&&& \\
\cdot & \cdot &&&&&&&&&&&& \\
\cdot & \cdot & \cdot &&&&&&&&&&& \\
\cdot & \cdot & \cdot & \cdot  &&&&&&&&&& \\
\cdot & \cdot & \cdot & \cdot  & \cdot &&&&&&&&& \\
\cdot & \cdot & \cdot & \star  & \cdot & \bullet &&&&&&&& \\
&&&&& \bullet & \bullet &&&&&&& \\
&&&&& \bullet & \bullet & \bullet &&&&&& \\
&&&&& \bullet & \bullet & \bullet & \cdot &&&&& \\
&&&&& \bullet & \bullet & \bullet & \cdot & \cdot &&&& \\
&&&&& \bullet & \bullet & \bullet & \cdot & \cdot & \cdot &&& \\
&&&&& \cdot & \cdot & \cdot &  \ast & \cdot & \cdot & \cdot && \\
&&&&& \cdot & \cdot & \cdot & \cdot & \cdot & \cdot & \cdot & \cdot & \\
&&&&& \cdot & \cdot & \cdot & \cdot & \cdot & \cdot & \cdot & \cdot & \cdot
\end{array}
$$
$$ Figure\ I3\ (\star \sim x_{k_1l_1};\ \ast\sim x_{k_2l_2};\ \bullet\sim x_{k_3l_3})$$
\emph{Subcase} (I4):
In this subcase we have the  following setting
$$N=2\ ; x_{k_2l_2}\prec x_{k_3l_3}\ ;\ x_{k_1l_1}(-j-1) \in\triangle_r\ ;\ x_{k_2l_2}(-j)\in{}\sp r\!\triangle\ .$$
Since the $x_{k_1l_1}(-j-1)x_{k_3l_3}(-j)$ is not leading term then $\sharp \ x_{k_3l_3} = k_1 -1$ (see Figure I4) and  the  sum $\sum_{I4}$ is
\begin{equation}\label{sumaI4 final}
\sum_{I4}=\sum_{i_1=1}^{2n}\sum_{j_1=1}^{i_1}\sum_{i_2=i_1}^{2n}\sum_{j_2=i_1}^{i_2} [i_1-1]\ .
\end{equation}

$$
\begin{array}{cccccccccccccc}
\cdot &&&&&&&&&&&&& \\
\cdot & \cdot &&&&&&&&&&&& \\
\cdot & \cdot & \cdot &&&&&&&&&&& \\
\cdot & \cdot & \cdot & \cdot  &&&&&&&&&& \\
\cdot & \cdot & \cdot & \cdot  & \cdot &&&&&&&&& \\
\cdot & \cdot & \cdot & \star  & \cdot & \cdot &&&&&&&& \\
&&&&& \cdot & \cdot &&&&&&& \\
&&&&& \cdot & \cdot & \cdot &&&&&& \\
&&&&& \cdot & \cdot & \cdot & \cdot &&&&& \\
&&&&& \cdot & \cdot & \cdot & \cdot & \cdot &&&& \\
&&&&& \cdot & \cdot & \cdot & \cdot & \cdot & \cdot &&& \\
\bullet&\bullet&\bullet&\bullet&\bullet& \cdot & \cdot & \cdot &  \ast & \cdot & \cdot & \cdot && \\
&&&&& \cdot & \cdot & \cdot & \cdot & \cdot & \cdot & \cdot & \cdot & \\
&&&&& \cdot & \cdot & \cdot & \cdot & \cdot & \cdot & \cdot & \cdot & \cdot
\end{array}
$$
$$ Figure\ I4\ (\star \sim x_{k_1l_1};\ \ast\sim x_{k_2l_2};\ \bullet\sim x_{k_3l_3})$$
\emph{Subcase} (I5):
Since in this subcase $x_{k_1l_1}(-j-1)x_{k_2l_2}(-j)$ is not leading term then we select  entirely free the  position of $x_{k_1l_1}(-j-1) \in\triangle_r$ and $x_{k_3l_3}(-j)\in{}\sp r\!\triangle$. Then the  corresponding  setting is
$$N=2\ ; x_{k_2l_2}\prec x_{k_3l_3}\ ;\ x_{k_1l_1}(-j-1) \in\triangle_r\ ;\ x_{k_3l_3}(-j)\in{}\sp r\!\triangle\ .$$
Since the $x_{k_1l_1}(-j-1)x_{k_2l_2}(-j)$ is not leading term then $\sharp \ x_{k_2l_2} = (2n-k_3)(k_1 -1)$ (see Figure I5) and  the  sum $\sum_{I5}$ is
\begin{equation}\label{sumaI4 final}
\sum_{I5}=\sum_{k_1=1}^{2n}\sum_{l_1=1}^{k_1}\sum_{k_3=k_1}^{2n}\sum_{l_3=k_1}^{k_3} (2n-k_3)(k_1 -1)\ .
\end{equation}

$$
\begin{array}{cccccccccccccc}
\cdot &&&&&&&&&&&&& \\
\cdot & \cdot &&&&&&&&&&&& \\
\cdot & \cdot & \cdot &&&&&&&&&&& \\
\cdot & \cdot & \cdot & \cdot  &&&&&&&&&& \\
\cdot & \cdot & \cdot & \cdot  & \cdot &&&&&&&&& \\
\cdot & \cdot & \cdot & \star  & \cdot & \cdot &&&&&&&& \\
&&&&& \cdot & \cdot &&&&&&& \\
&&&&& \cdot & \cdot & \cdot &&&&&& \\
&&&&& \cdot & \cdot & \cdot & \cdot &&&&& \\
&&&&& \cdot & \cdot & \bullet & \cdot & \cdot &&&& \\
\ast&\ast&\ast&\ast&\ast&  \cdot & \cdot & \cdot & \cdot & \cdot & \cdot &&& \\
\ast&\ast&\ast&\ast&\ast& \cdot & \cdot & \cdot &  \cdot & \cdot & \cdot & \cdot && \\
\ast&\ast&\ast&\ast&\ast&  \cdot & \cdot & \cdot & \cdot & \cdot & \cdot & \cdot & \cdot & \\
\ast&\ast&\ast&\ast&\ast&  \cdot & \cdot & \cdot & \cdot & \cdot & \cdot & \cdot & \cdot & \cdot
\end{array}
$$
$$ Figure\ I5\ (\star \sim x_{k_1l_1};\ \ast\sim x_{k_2l_2};\ \bullet\sim x_{k_3l_3})$$
Finally we have
$$\sum_{I1}+\sum_{I2}+\sum_{I3}+\sum_{I4}+\sum_{I5} = 2n
{2n+4\choose 5}\ .$$\smallskip

In other two cases counting of embeddings of leading terms is similar and shows that (\ref{sum P^3(n)}) holds.
\end{proof}

\section{Combinatorial Rogers-Ramanujan type identities}

As a consequence of Theorem \ref{T: a basis of basic module}  we have a combinatorial Rogers-Ramanujan
type identities by using Lepowsky's product formula for principaly specialized characters of $C_n\sp{(1)}$-modules
$L(\Lambda_0)$ (see \cite{L} and \cite{M}, cf. \cite{MP2} for $n=1$)
\begin{equation}\label{E: conjecture 2}
\prod_{\substack{j\geq 1\\j\not\equiv 0\,\textrm{mod}\,2}}\!\!\frac{1}{1-q^j}
\prod_{\substack{j\geq 1\\j\not\equiv 0,\pm1\,\textrm{mod}\,n+2}}\!\!\frac{1}{1-q^{2j}}.
\end{equation}
This product can be interpreted combinatorially as a generating function for number of partitions
\begin{equation}\label{E: partition of m}
N=\sum_{m\geq 1}mf_{m}.
\end{equation}
of $N$ with parts $m$ satisfying  congruence conditions.
\begin{equation}\label{E: conjecture 4}
f_m=0\quad\text{if}\quad  m\equiv 0, \pm 2\,\textrm{mod}\,2n+4.
\end{equation}

On the other hand, in the principal specialization $e\sp{-\alpha_i}\mapsto q\sp1$, $i=0,1,\dots, n$, the sequence
of basis elements in $C_n\sp{(1)}$
\begin{equation}\label{E: conjecture 5}
X_{ab}(-1), \ ab\in B,\quad X_{ab}(-2),\  ab\in B,\quad X_{ab}(-3), \ ab\in B,\quad \dots
\end{equation}
obtains degrees
$$
|X_{ab}(-j)|=a+b-1+2n(j-1),
$$
where we prefer row and column indices of basis elements $X_{ab}\in B$ to be natural numbers
$$
b=1,\dots, 2n, \qquad a=1,\dots, b.
$$
For example, the basis elements for $C_2\sp{(1)}$ in the sequence (\ref{E: conjecture 5}) obtan degrees
\begin{equation}\label{E: conjecture 6}
\begin{array}{cccc}
1 &  &&   \\
2 & 3 & &  \\
3 & 4 & 5 & \\
4 & 5 & 6 & 7  \\
\end{array}
\begin{array}{cccc}
5 &  &&   \\
6 & 7 & &  \\
7 & 8 & 9 & \\
8 & 9 & 10 & 11  \\
\end{array}
\begin{array}{cccc}
9  &  &&   \\
10 & 11& &  \\
11 &  12& 13 & \\
12& 13 & 14& 15  \\
\end{array}
\dots
\end{equation}
As we see, there are several basis elements of a given degree $m$,
$$
m=a+b-1+2n(j-1),
$$
so we make them ``distinct'' by assigning to each degree $m$ a ``color'' $b$, the row index in which $m$ appears:
$$
m_b,\qquad |m_b|=m.
$$
For example, for $n=2$ we have
\begin{equation}
\begin{array}{cccc}
1_1 &  &&   \\
2_2 & 3_2 & &  \\
3_3 & 4_3 & 5_3 & \\
4_4 & 5_4 & 6_4 & 7_4  \\
\end{array}
\begin{array}{cccc}
5_1 &  &&   \\
6 _2& 7_2 & &  \\
7_3 & 8_3 & 9_3 & \\
8_4 & 9_4 & 10_4 & 11_4  \\
\end{array}
\begin{array}{cccc}
9_1  &  &&   \\
10 _2 & 11 _2& &  \\
11_3 &  12_3& 13_3 & \\
12_4& 13_4 & 14_4& 15_4  \\
\end{array}
\dots,
\end{equation}
so that numbers in the first row have color 1, numbers in the second row have color 2, and so on.
In general we consider a disjoint  union $\mathcal D_n$  of integers in $2n$ colors, say $m_1,  m_2, \dots,  m_{2m}$,
satisfying the congruence conditions
\begin{equation}\label{E: conjecture 8}
\begin{aligned}
&\{m_1\mid m\geq 1, m\equiv 1\,\textrm{mod}\,2n\}, \\
&\{m_2\mid m\geq 2, m\equiv 2,3\,\textrm{mod}\,2n\}, \\
&\{m_3\mid m\geq 3, m\equiv 3,4,5\,\textrm{mod}\,2n\}, \\
&\qquad\dots\\
&\{m_{2n}\mid m\geq 2n, m\equiv 2n,2n+1,\dots, 4n-1\,\textrm{mod}\,2n\} \\
\end{aligned}
\end{equation}
and arranged in a sequence of triangles.

For fixed $m$ and $b$ parameters $a$ and $j$ are completely determined. We see this easily for the last row
$$
2n_{2n},\dots,(4n-1)_{2n}; 4n_{2n},\dots,(6n-1)_{2n}; 6n_{2n},\dots \ ,
$$
and then for all the other rows as well. So instead of $m_b$ we may write $m_{ab}(-j)$.

\begin{theorem} For every positive integer $N$ the number of partitions
$$
N=\sum_{m\geq 1}mf_{m}
$$
with congruence conditions \
$f_m=0$ \ if \ $m\equiv 0, \pm 2\,\textrm{mod}\,2n+4$ \ equals the number of colored partitions
\begin{equation}\label{E: colored partitions}
N=\sum_{m_b\in\mathcal D_n}|m_b|f_{m_b}
\end{equation}
with difference conditions
\ $f_{m_b}+f_{m'_{b'}}\leq 1$ \ if
\begin{itemize}
\item  $m_b=m_{ab}(-j-1)$ and $m'_b=m'_{a'b'}(-j)$ such that $b\geq a'$, or
\item $m_b=m_{ab}(-j)$ and $m'_b=m'_{a'b'}(-j)$ such that $b\leq b'$, $a\geq a'$.
\end{itemize}
\end{theorem}

For adjacent triangles corresponding to
$$
\dots,\quad X_{ab}(-j), ab\in B,\quad X_{ab}(-j-1), ab\in B,\quad\dots
$$
in  (\ref{E: conjecture 5}) and a fixed row $r\in\{1,\dots,2n\}$ we consider the corresponding two triangles: ${}\sp r\!\triangle$
on the left and $\triangle_r$
on the right. For example, for $n=2$ and the third row we have $r=3$ and two triangles denoted by bullets:
\begin{equation}\label{E: conjecture 9}
\dots
\begin{array}{cccc}
\cdot &  &&   \\
\cdot & \cdot & & \\
\cdot &  \cdot&  \cdot&  \\
 \cdot &  \cdot& \cdot & \cdot  \\
\end{array}
\begin{array}{cccc}
\cdot &  &&   \\
\cdot & \cdot & & \\
\cdot &  \cdot&  \bullet  &  \\
 \cdot &  \cdot& \bullet  & \bullet   \\
\end{array}
\begin{array}{cccc}
\bullet  &  &&   \\
\bullet  & \bullet  & & \\
\bullet  &  \bullet  &  \bullet  &  \\
 \cdot &  \cdot& \cdot & \cdot  \\
\end{array}
\begin{array}{cccc}
\cdot &  &&   \\
\cdot & \cdot & & \\
\cdot &  \cdot&  \cdot&  \\
 \cdot &  \cdot& \cdot & \cdot  \\
\end{array}
\dots
\end{equation}
are \ ${}\sp{3}\triangle$ on the left and $\triangle_{3}$ on
the right.

Then the first difference condition {\bf does not allow} two parts
in a colored partition (\ref{E: colored partitions}) such that
$$
m'_b=m'_{a'b'}(-j)\in {}\sp r\!\triangle\quad\text{and}\quad m_b=m_{ab}(-j-1)\in \triangle\sb r.
$$
On the other hand, the second difference condition {\bf does not allow} two parts
in a colored partition (\ref{E: colored partitions}) such that
$$
m'_b=m'_{a'b'}(-j),\quad m_b=m_{ab}(-j)
$$
to be in any rectangle such as:
\begin{equation}\label{E: conjecture 10}
\dots
\begin{array}{cccc}
\cdot &  &&   \\
\cdot & \cdot & & \\
\cdot &  \cdot&  \cdot&  \\
 \cdot &  \cdot& \cdot & \cdot  \\
\end{array}
\begin{array}{cccc}
\cdot &  &&   \\
\cdot & \cdot & & \\
\bullet &  \bullet&  \bullet  &  \\
 \bullet &  \bullet& \bullet  & \cdot   \\
\end{array}
\begin{array}{cccc}
\cdot  &  &&   \\
\cdot  & \cdot  & & \\
\cdot  &  \cdot  &  \cdot  &  \\
 \cdot &  \cdot& \cdot & \cdot  \\
\end{array}
\begin{array}{cccc}
\cdot &  &&   \\
\cdot & \cdot & & \\
\cdot &  \cdot&  \cdot&  \\
 \cdot &  \cdot& \cdot & \cdot  \\
\end{array}
\dots
\end{equation}


\end{document}